\documentclass{amsart}
\usepackage{amssymb}
\usepackage{setspace}
\newtheorem{theorem}{Theorem}[section]
\newtheorem*{theorem*}{Theorem}

\newtheorem{proposition}[theorem]{Proposition}
\newtheorem{corollary}[theorem]{Corollary}

\theoremstyle{remark}
\newtheorem*{remark*}{Remark}

\newtheorem*{example*}{Example}

\newcommand{\groundfield}{\mathbf{F}}
\newcommand{\aone}{\mathbf{A}^1}
\newcommand{\aonecross}{\mathbf{A}^{1\times}}
\newcommand{\Pic}{\mathrm{Pic}}
\newcommand{\Picz}{\mathrm{Pic}^0}
\newcommand{\NS}{\mathrm{NS}}
\newcommand{\gm}{\mathbf{G}_m}
\newcommand{\codim}{\mathrm{codim}~}

\begin{document}
\title[Algebraic Degeneracy]{Algebraic Degeneracy
of Non-Archimedean Analytic Maps}
\author{Ta Thi Hoai An}
\address{Institute of Mathematics\\
18 Hoang Quoc Viet, Cau Giay\\
Hanoi, VIETNAM}
\email{tthan@math.ac.vn}
\thanks{Partial financial support for the first author provided to her
as a Junior Associate Member of ICTP}
\author{William Cherry}
\thanks{Partial financial support for the second author provided
by the United States National Security Agency under
Grant Number H98230-07-1-0037.
The United States Government is authorized to reproduce and
distribute reprints not­withstanding any copyright notation herein.}
\address{Department of Mathematics\\University of North Texas\\ 
P.O. Box 311430, Denton, TX  76203\\USA}
\email{wcherry@unt.edu}
\author{Julie Tzu-Yueh Wang}
\address{Institute of Mathematics\\Academia Sinica\\Nankang\\
Taipei 11529\\TAIWAN}
\email{jwang@math.sinica.edu.tw}
\date{July 20, 2007}
\begin{abstract}
We prove non-Archimedean analogs of results of
Noguchi and Winkelmann showing algebraic degeneracy of 
rigid analytic maps to projective varieties
omitting an effective divisor with
sufficiently many irreducible components relative to the rank of the group
they generate in the N\'eron-Severi group of the variety.
\end{abstract}
\maketitle
\section{Introduction}
Dufresnoy \cite{Dufresnoy} (or see \cite[\S3.10]{Kobayashi})
proved that a holomorphic map from the complex plane $\mathbf{C}$
that omits $n+k$ hyperplanes in general position in projective space
$\mathbf{P}^n$
must be contained in a linear subspace of dimension at most 
$n/k.$
Noguchi and Winkelmann \cite{NoguchiWinkelmann}
generalized this result to show that a holomorphic curve
in an arbitrary projective manifold (or more generally a compact
K\"ahler manifold) omitting sufficiently many irreducible hypersurfaces
relative to the rank of the group generated by their cohomology classes
must be algebraically degenerate.  This puts the Dufresnoy theorem
in context and clarifies the roll played by the rank of the
N\'eron-Severi group.  Noguchi and Winkelmann's precise result
is:

\begin{theorem*}[Noguchi/Winkelmann]
Let $M$ be a compact K\"ahler manifold of dimension $m.$
Let $\{D_i\}_{i=1}^\ell$ be $\ell$ irreducible hypersurfaces
in general position. Let $r$ be the rank of the group generated
by $\{c_1(D_i)\}_{i=1}^\ell.$
Let $W$ be a closed subvariety of $M$
of dimension $n$ and irregularity $q.$
Suppose there exists an algebraically non-degenerate holomorphic
map from the complex plane $\mathbf{C}$ to $W$ that omits each of
the $D_i$ that does not contain all of $W.$ Then
\begin{enumerate}
\item[(i)] $\#\{W\cap D_i \ne W\}+q \le n + r;$
\item[(ii)] If $\ell>m$ and in addition each of the $D_i$ are ample, then
$$
	n \le \frac{m}{\ell-m} \max\{0,r-q\}.
$$
\end{enumerate}
\end{theorem*}

Here, $c_1$ denotes the first Chern class,
and the irregularity $q$ is the dimension of the space of holomorphic
$1$-forms on a desingularization of $W,$
which is the dimension of the Albanese variety. When the irregularity
exceeds the dimension, Bloch \cite{Bloch} (or see
\cite{Ochiai} for a rigorous proof)
proved all holomorphic curves are algebraically degenerate
by showing the image in the Albanese variety is degenerate.  
This was extended by Noguchi
(and by Noguchi/Winkelmann in the non-algebraic K\"ahler case) --
see the references in \cite{NoguchiWinkelmann} -- to conclude
that a holomorphic map from $\mathbf{C}$ omitting an effective divisor $D$
such that the logarithmic irregularity with respect to $D,$
\textit{i.e.,} the dimension of the space of logarithmic $1$-forms with
poles along $D,$ exceeds the dimension.  This was again done by
composing with the quasi-Albanese morphism, but this time since the
quasi-Albanese variety need not be compact, there are additional
difficulties.

Our purpose here is to explain the rigid analytic analog of 
the Noguchi/Winkel\-mann theorem for non-Archimedean analytic maps,
at least in the case of projective
algebraic varieties.  The non-Archimedean analog of
Bloch's theorem was proven by Cherry in \cite{Cherry94}.
Modulo the standard constructions of Albanese and Picard varieties
as in \cite{LangAV},
it is then a routine matter to conclude analogous algebraic degeneracy
results.  
\par
One could perhaps argue that the natural category for us to work in
is that of complete rigid analytic spaces, rather than projective varieties.
However, the Albanese map is crucial for our arguments, and 
as far as we know, this has not been worked out for non-algebraic
rigid analytic spaces. Since we restrict ourselves to projective varieties,
we do not hesitate to appeal to algebraic results when convenient, even
when there are alternative approaches that work for more general rigid
analytic spaces.
\par
Let $\groundfield$ be an algebraically closed field complete with
respect to a non-Archimedean absolute value and of arbitrary
characteristic.  Throughout, a variety
will mean an algebraic variety defined over $\groundfield,$ and
a morphism will mean an algebraic morphism defined over $\groundfield.$
We will use $\aone$ to denote the affine line over $\groundfield$
and $\aonecross$ to denote $\aone\setminus\{0\}.$
We will use $\gm$ to denote the multiplicative group, which as
an analytic space or a variety is of course the same thing as
$\aonecross.$ Analytic will mean rigid analytic over $\groundfield.$  
One can think concretely by thinking of an analytic map from
$\aone$ (resp.\ $\aonecross$) to an
algebraic variety $X$ as given by a solution to the defining equations
of $X$ in formal power series (resp.\ formal Laurent series)
with coefficients in $\groundfield$ and
converging for arbitrary positive radii.

\section{Non-Archimedean Analytic Maps to Semi-Abelian Varieties}

We begin by recalling the main results of \cite{Cherry94}
and extending them to semi-Abelian varieties.

\begin{theorem}\label{savhyp}
Any analytic map from $\aone$ to a semi-Abelian variety
must be constant.
\end{theorem}

\begin{proof}
The proof is essentially the argument on page~401 of \cite{Cherry94}.
What was actually shown in \cite{Cherry94} was that an analytic map
into the extension of an Abelian variety with good reduction by a 
torus must be constant.  This was then used with the semi-Abelian
uniformization theorem to show that an analytic map to an arbitrary
Abelian variety must be constant.  Once one has this, the same argument
repeated then gives that an analytic map from $\aone$
to an arbitrary semi-Abelian
variety must be constant.
\end{proof}

In characteristic zero, Theorem~\ref{savhyp} can also be proven by
the method of \cite{CherryRu}.

\begin{corollary}\label{savdegen}
If $X$ is a variety admitting a
non-constant morphism to a semi-Abelian variety, then any analytic
map from $\aone$ to $X$ is algebraically degenerate.
\end{corollary}

\begin{corollary}\label{posq}
If $X$ is a non-singular projective variety in characteristic zero with
positive irregularity, then any analytic map from $\aone$ to 
$X$ must be algebraically degenerate.
\end{corollary}

For the case of surfaces, Corollary~\ref{posq} was pointed out in
\cite{Cherry93}.

\begin{proof}
Since we assume $X$ to be non-singular, the Albanese map is a morphism,
see \textit{e.g.,} \cite[Ch.~2]{LangAV},
and since we have assumed characteristic
zero, the dimension of the Albanese variety is the same as the irregularity.
\end{proof}

\begin{corollary}\label{ratabdegen}
If $X$ is a projective variety admitting a non-constant rational map
to an Abelian variety, then any analytic map from $\aone$ to $X$
is algebraically degenerate.
\end{corollary}

\begin{proof}
Let $\phi:X\rightarrow A$ be a non-constant rational map to an Abelian
variety and let $f$ be an analytic map from $\aone$ to $X.$
The idea is that $\phi\circ f$ is a meromorphic mapping from
$\aone$ to $A$ and hence analytic, unless $f$ is contained
in the indeterminacy locus of $\phi,$ whence degenerate.  The corollary then
follows from the theorem.
\par
Lacking a convenient reference
for the general fact, we give an ad-hoc proof here.
Embed $X$ and $A$ in projective spaces.
Then $\phi$ can be represented as
[$\phi_0,...,\phi_N$] where the $\phi_i$ are homogeneous polynomials
and $N$ is the dimension of the projective space in which we have 
embedded $A.$  Then, $\phi_i\circ f$ are analytic functions on 
$\aone.$  If they are all identically zero, then the image of $f$
is contained in the indeterminacy locus of $\phi$ and is algebraically
degenerate.  Otherwise, factoring out the greatest common divisor
(well-defined up to a non-zero constant)
from the $\phi_i\circ f$ if necessary, we see $\phi\circ f$ can be made to be
well-defined on all of $\aone.$ Hence, $\phi\circ f$ is an analytic
map to $A$
and hence constant by the theorem.  Therefore, $f$ is algebraically 
degenerate.
\end{proof}

\begin{proposition}\label{gmprodprop}
Let $T$ be a multiplicative torus and let $f$ be an analytic map
(not assumed to be a group homomorphism) from $T$ to $\gm.$
Then $f$ is the translation of a group homomorphism.
\end{proposition}

\begin{proof}
Embed $T$ in affine $n$-space $\mathbf{A}^n$ in the natural way
with  affine coordinates \hbox{$z=(z_1,\dots,z_n)$}
on $\mathbf{A}^n.$
Then, $f$ can be written as a Laurent series in multi-index notation as
$$
	\sum_{\gamma\in\mathbf{Z}^n}a_\gamma z^\gamma.
$$
An easy argument involving valuation polygons shows that
there exists precisely one multi-index $\gamma$ such that
$a_\gamma\ne0,$ from which the proposition follows.
\par
Indeed, suppose there exist two multi-indices 
$$
	\mu=(\mu_1,\dots,\mu_n)\ne(\nu_1,\dots,\nu_n)=\nu
$$
with $a_\mu\ne0$ and $a_\nu\ne0.$  Then, there must be some $k$
such that $\mu_k\ne\nu_k.$  Without loss of generality
by reordering the coordinates if necessary, assume $\mu_n\ne\nu_n.$
Let $u=(u_1,\dots,u_{n-1})$ be such
that $|u_j|=1.$ Let $\tilde u = (\tilde u_1,\dots, \tilde u_{n-1})$
be the reduction of $u$ in $\mathbf{A}^{n-1}(\widetilde{\groundfield}),$ where 
$\widetilde{\groundfield}$ denotes the residue class field of $\groundfield.$
We want to make a substitution of the form 
$z_n=z$ and $z_j=u_j$ for $1\le j \le n-1$ to get a Laurent series
in one variable $z$ with at least two non-zero coefficients, but we
need to choose the $u_j$ so as not to have any accidental cancellation.
But clearly if we choose $u$ with $\tilde u$ generic, meaning that there is a 
non-zero polynomial with coefficients in $\widetilde{\groundfield}$
such that if $\tilde u$ is not in the zero locus of that polynomial,
then
$$
	\left|\sum_{\gamma \textnormal{~s.t.~}\gamma_n=\mu_n}
	\!\!\!\!\!\!a_\gamma u_1^{\gamma_1}\dots u_{n-1}^{\gamma_{n-1}}\right|
	= \sup\{|a_\gamma| : \gamma_n=\mu_n\}
	\ge |a_\mu|\ne0
$$
and
$$
	\left|\sum_{\gamma \textnormal{~s.t.~}\gamma_n=\nu_n} 
	\!\!\!\!\!\!a_\gamma u_1^{\gamma_1}\dots u_{n-1}^{\gamma_{n-1}}\right|
	= \sup\{|a_\gamma| : \gamma_n=\nu_n\}
	\ge |a_\nu|\ne0,
$$
and so we will indeed get a a one-variable Laurent series with
at least two non-zero coefficients. This contradicts
the assumption that $f$ is zero-free by the theory of valuation
polygons, and so there could have only
been one multi-index $\gamma$ with $a_\gamma\ne0.$
\end{proof}

\begin{theorem}\label{savhom}
Let $f:\aonecross\rightarrow S$ be an analytic map to a semi-Abelian
variety $S.$  Then $f$ is the translate of a group homomorphism from
$\gm$ to $S.$
\end{theorem}

\begin{proof}
By composing with a translation of $S,$ we may assume 
without loss of generality that
$f(1)=1.$ Let $\phi:S\rightarrow A$ be the homomorphism defining $S$
as an extension (by a multiplicative torus) of the Abelian variety $A.$
By the argument on page~401 of \cite{Cherry94}, $\phi\circ f$
is a group homomorphism from $\gm$ to $A.$
Now let \hbox{$(z_1,z_2)\in\gm\times\gm$} and consider
the analytic map $\psi$ from $\gm\times\gm$ to $S$ defined by
$$
	\psi(z_1,z_2)=f(z_1z_2)[f(z_1)]^{-1}[f(z_2)]^{-1}.
$$
Because $\phi$ and $\phi\circ f$ are group homomorphisms,
$\phi\circ\psi$ is the constant map from $\gm\times\gm$
to the identity element of $A.$  Hence $\psi$ can be thought of
as an analytic map to a multiplicative torus $T.$  It then follows
from Proposition~\ref{gmprodprop} by projecting from $T$ onto each
of its factors that $\psi$ is the translation of a group homomorphism.
But, since $\psi(1,1)=1,$ we have that $\psi$ is in fact a group
homomorphism. Clearly, \hbox{$\psi(z_1,1)=\psi(1,z_2)=1,$}
and hence $\psi$ is the constant map to the identity.  In other words,
$f$ is a group homomorphism.
\end{proof}

\begin{theorem}\label{savsubv}
Let $X\subset S$ be a closed subvariety of a semi-Abelian variety $S.$
Let \hbox{$f:\aone\rightarrow X$} be a non-constant analytic map.  Then,
the image of $f$ is contained in the translate of a non-trivial
semi-Abelian subvariety of $S$ contained in $X.$
\end{theorem}

\begin{proof}
By Theorem~\ref{savhom}, the map $f$ is the translate of a group
homomorphism.  Thus, by \cite[p.~84]{LangHyp},
the Zariski closure of the image of $f$ is 
the translate of a subgroup of $S.$
\end{proof}

\begin{corollary}[Non-Archimedean Bloch theorem \cite{Cherry94}]
\label{blochcor}
If $X$ is a non-singular projective variety whose Albanese variety
has dimension larger than $\dim X,$ then any analytic map from
$\aonecross$ to $X$ is algebraically degenerate.
\end{corollary}

\begin{remark*} In \cite{Cherry94}, Corollary~\ref{blochcor}
was incorrectly stated in terms of the irregularity
$q=\dim H^1(X,\mathcal{O}_X)$ rather than the dimension of the
Albanese variety.  In characteristic zero, both numbers are equal,
but an example of Igusa \cite{Igusa} shows that in positive
characteristic the dimension of the Albanese variety can be smaller
than the dimension of the space of regular $1$-forms.
\end{remark*}

\section{Algebraic Degeneracy of non-Archimedean Analytic Maps
Omitting Sufficiently Many Divisors}

In this section we will develop the non-Archimedean analogs
of the work of Noguchi and Winkelmann \cite{NoguchiWinkelmann}.
\par
Let $X$ be a projective variety non-singular in codimension one.
Recall that the space $\Pic(X)$ classifies
Cartier divisors on $X$ up to linear equivalence.  
Those divisor classes in $\Pic(X)$ which are algebraically
equivalent to zero are denoted by $\Pic^0(X),$
and $\Pic^0(X)$ is an Abelian variety, known as the Picard variety
which is dual to the Albanese variety, see \textit{e.g.,} \cite{LangAV}.
The quotient $\Pic(X)/\Picz(X)$ is a finitely generated group
called the N\'eron-Severi group of $X$ and denoted $\NS(X).$
Finite generation of $\NS(X)$ is a theorem of Severi in characteristic
zero and N\'eron in positive characteristic.  It also follows from the
Lang-N\'eron theorem, see \textit{e.g.,} \cite{LangFDG},
or by \'etale cohomology, see \textit{e.g.,} \cite{Milne}.
We will refer to the canonical image of a divisor $D$ in $\NS(X)$
as the \textit{Chern class} of the divisor and denote it by $c_1(D).$
In characteristic zero when $X$ is non-singular,
this agrees with the classical notion of
Chern class in $H^2(X,\mathbf{Z})$ coming from the exponential sheaf
sequence.  In positive characteristic, one can embed $\NS(X)$
in an \'etale cohomology group, see \textit{e.g.,} \cite{Milne},
and think of $c_1$ that way. For
us, the cohomological interpretation will not be important, so we prefer
to simply think of $c_1$ as a homomorphism from divisors to 
$\NS(X).$
\par
If $\iota:Y\rightarrow X$ is a morphism (or more generally
a rational map) from a projective variety $Y,$
non-singular in codimension one, to a non-singular projective variety
$X,$ then if $D$ is
in $\Picz(X),$ then $\iota^*D$ is in $\Picz(Y)$ by
\cite[Ch.~V, Prop.~1]{LangAV}. Hence, the pull-back map on divisor classes
$\iota^*:\Pic(X)\rightarrow\Pic(Y)$ induces a homomorphism
$\iota^*:\NS(X)\rightarrow\NS(Y).$
\par
We begin by discussing analytic maps from $\aone$ omitting sufficiently
many divisors relative to the size of the group generated by their 
Chern classes.

\begin{theorem}\label{aonethm}
Let $Y$ be a possibly singular projective variety
and let $\iota:Y\rightarrow X$ be a morphism to a smooth projective
variety $X.$
Let $\{D_i\}_{i=1}^\ell$ be $\ell$ irreducible effective divisors on $X$
such that $\{\iota^* D_i\}_{i=1}^\ell$ form $\ell$ distinct effective
Cartier divisors on $Y.$ 
Assume the number of irreducible components $\ell$ is larger than 
the rank of the subgroup generated by the $c_1(D_i)$ in 
$\NS(X).$  Then, any analytic map from 
$\aone$ to $Y$ is either algebraically degenerate or
intersects the support of at least one of the $\iota^*D_i.$
\end{theorem}

Any $r$ algebraically independent entire functions form
an algebraically non-degenerate analytic map from $\aone$
to $Y=X=(\mathbf{P}^1)^r$ that omits the $r$ divisors defined by 
taking the point at $\infty$ on one of the $\mathbf{P}^1$ factors
and thus show the theorem is optimal in its dependence on
the rank of the group generated by the $c_1(D_i).$

Typically what we have in mind for $Y$ is 
a closed subvariety of $X.$  In that case the map $\iota$ is
the inclusion in $X,$
and $\iota^* D_i$ is set-theoretically $D_i\cap Y.$

Notice that unlike 
\cite{NoguchiWinkelmann}, we do not need to make any kind of general
position assumption on the $D_i,$ other than that the
$\iota^*D_i$ are distinct.

In the case that $Y=X=\mathbf{P}^n,$ we recover the well-known trivial fact
that a non-Archimedean analytic map from $\aone$ to $\mathbf{P}^n$
that omits two distinct hypersurfaces is algebraically degenerate.

\begin{proof}
Let $f:\aone\rightarrow Y$ be an algebraically non-degenerate
analytic map.
\par
Let $\widetilde{Y}$ be the normalization of $Y,$ which of course is
non-singular in codimension one.  Let
$\tilde \iota:\widetilde{Y}\rightarrow X$ denote the composition of the
natural map from $\widetilde{Y}$ to $Y$ with
$\iota:Y\rightarrow X.$
\par
If $f$ is not algebraically degenerate, then $Y$ is not contained
in the indeterminacy locus of the rational map from $Y$
to $\widetilde{Y},$ and hence lifts (as in the proof of
Corollary~\ref{ratabdegen}) to an analytic map
$\tilde f$ from $\aone$ to $\widetilde{Y}$
\par
By our assumption that there are more components $D_i$ than the rank
of the group the $c_1(D_i)$
generate in $\NS(X),$ we can find integers $a_i$ not
all zero so that $\sum a_i c_1(D_i)=0.$ Thus, 
\hbox{$\sum a_i c_1(\tilde\iota^*D_i)
= \tilde\iota^*\left(\sum a_i c_1(D_i)\right)=0$}
in $\NS(\widetilde{Y}),$
and thus $\sum a_i \tilde\iota^*D_i$  is algebraically equivalent to
zero on $\widetilde{Y}.$  Because $\widetilde{Y}$ maps onto $Y,$
by our assumption that the $\iota^*D_i$ are distinct, we also have
that the $\tilde\iota^*D_i$ are distinct.
Also, because not all the $a_i$ are zero,
we conclude that $\sum a_i \tilde\iota^*D_i$ is not the zero divisor on 
$\widetilde Y.$
\par
If there is a non-constant rational map from $Y$
to an Abelian variety, then
$f$ is already algebraically degenerate by Corollary~\ref{ratabdegen}.
Thus, without loss of generality, we may assume there are no non-constant
rational maps from $\widetilde{Y}$ to Abelian varieties, or in other words
that the Albanese variety of $\widetilde{Y}$ is trivial.
Because the Picard variety $\mathrm{Pic}^0(\widetilde{Y})$
is Cartier dual to the Albanese variety, $\mathrm{Pic}^0(\widetilde{Y})$
is also trivial. 
But $\mathrm{Pic}^0(\widetilde{Y})$ is precisely the set of divisors
algebraically equivalent to zero modulo those divisors linearly equivalent
to zero. Hence, every divisor algebraically equivalent to zero on
$\widetilde{Y}$ is also linearly equivalent to zero.  Thus, we can find
a non-constant rational function $h$ on $\widetilde{Y}$ such that
$$
	\mathrm{div}(h)=\sum a_i \tilde\iota^* D_i.
$$

If $f$  omits the supports of all the $\iota^*D_i,$ then its lift
$\tilde f: \aone\rightarrow \tilde Y$ is an  
analytic map  omitting the supports of all the $\tilde\iota^*D_i.$
Then, $h\circ \tilde f$ is an analytic map from $\aone$ to
$\aonecross,$ and hence constant.
Thus, $\tilde f$ is algebraically degenerate and so is  $f$.
\end{proof}

Next, we recall that a collection of irreducible effective ample divisors
$D_i$ in a non-singular projective variety $X$ of dimension $m$ 
are said to be \textit{in general position} if for each 
\hbox{$1\le k \le m+1$} and each choice of indices
\hbox{$i_1<\dots<i_k,$} each irreducible component of
$$
	D_{i_1}\cap\dots\cap D_{i_k}
$$
has codimension $k$ in $X,$ so in particular is empty when $k=m+1.$

\begin{corollary}\label{degencor}
Let $Y$ be a closed positive dimensional 
subvariety of a non-singular projective
variety $X.$  Let $\{D_i\}_{i=1}^\ell$ be $\ell$ irreducible,
effective, ample divisors in general position on $X.$ 
Let $r$ be the rank of the subgroup of $\NS(X)$
generated by $\{c_1(D_i)\}_{i=1}^\ell.$
If there exists an algebraically non-degenerate analytic map
from $\aone$ to $Y$ omitting each of the $D_i$ that does not
contain all of $Y,$
then 
$$
	\ell\le\max\left\{r+\codim Y, r\cdot\frac{\dim X}{\dim Y}\right\}.
$$
\end{corollary}

When $r=1,$ the term $r+\codim Y=1+\codim Y$ is largest, and the
example of $Y$ a linear subspace of $X=\mathbf{P}^n$ shows the inequality
is optimal.  We do not have examples to show optimality when $r>1,$
and we suspect the inequality may not be optimal in that case.
When $Y\subset X=\mathbf{P}^n,$ Corollary~\ref{degencor} was proven
by An, Wang, and Wong \cite{AnWangWong}.

\begin{corollary}\label{hypcor}
Let $X$ be a non-singular projective
variety.  Let $\{D_i\}_{i=1}^\ell$ be $\ell$ irreducible,
effective, ample divisors in general position on $X.$ 
Let $r$ be the rank of the subgroup of $\NS(X)$
generated by $\{c_1(D_i)\}_{i=1}^\ell.$
Let $f$ be an analytic map from $\aone$ to $X$ omitting each of the $D_i.$
Then the image of $f$ is contained in an algebraic subvariety
$Y$ of $X$ such that
$$
\dim Y\le \max\{r+\dim X-\ell,
\frac{r\cdot\dim X}{\ell}\}.
$$
In particular, if
$$
	\ell\ge\max\{r+\dim X,
r\cdot \dim X +1\},
$$
then $f$ is constant.
\end{corollary}

Note that when $Y=X=\mathbf{P}^n,$ the fact that an analytic 
map from $\aone$ 
omitting $n+1$ hypersurfaces in general position must be constant
also follows from Ru's defect inequality \cite{RuHypersurface}.
The fact that an analytic map from $\aone$ to a projective
variety $X\subset\mathbf{P}^N$ omitting $\dim X+1$ hypersurfaces
of $\mathbf{P}^N$ in general position with $X$
is a consequence of An's defect inequality \cite{AnProc}.

\begin{proof}[Proof of Corollary \ref{degencor}]
Suppose $f$ is an algebraically non-degenerate analytic map
from $\aone$ to $Y$ omitting the $D_i.$  
Let $l_0$ be the cardinality of the set
$$
	\{D_i \cap Y : D_i\not\supset Y\},
$$
and note that $D_i\cap Y\ne\emptyset$ for all $i$ because the
$D_i$ are assumed ample. By the theorem,  
\begin{align}\label{l0}
\ell_0\le r.
\end{align}

We now estimate $l_0$ as in \cite{NoguchiWinkelmann}.
Let $n=\dim Y$.  Without loss of generality we may assume that
$D_1 \cap Y, D_2\cap Y,...,D_{\ell_0}\cap Y$ are distinct.  For $1\le j\le
\ell_0$, let $s_j$ be the number of divisors $D_i$ with
$D_i\cap Y=D_j\cap Y$.  Rearranging the indices, we may assume that
\begin{align}\label{sj}
s_1\ge s_2\ge ...\ge s_{\ell_0}.
\end{align}
\par
We first consider the case when $\ell_0\le n$.
Because the $D_i$ are ample and by the definition of $\ell_0,$ we have
$$
\emptyset\ne Y\cap\left(\bigcap_{j=1}^{\ell_0} D_j\right)
=Y\cap\left(\bigcap_{j=1}^{\ell} D_j\right)
$$
Because the $D_i$ are in general position, this implies
$$
	\dim Y-\ell_0\le \dim X-\ell,
$$
and hence
$$
	\ell\le \ell_0+\codim Y \le r+\codim Y
$$
by (\ref{l0}).
\par
The remaining case is $\ell_0> n$. 
Again, because the $D_i$ are ample, 
$$
	\emptyset \ne Y\cap\left(\bigcap_{j=1}^{n} D_j\right)
	=Y\cap\left(\bigcap_{i\in I} D_i\right),
$$
where $I=\{i : D_i\supset Y, \text{ or } D_i\cap Y=D_j\cap Y \text{ for
some } 1\le j\le n
\}
$.
Let $s_0$ denote the number of divisors $D_i$ such that 
\hbox{$D_i\supset Y.$} Since the divisors are 
in general position, this implies
\begin{align}\label{sum}
\sum_{i=0}^n s_i=\#I\le \dim X.
\end{align}
On the other hand, it follows from (\ref{sj}) that
$$
\frac1{\ell_0}\sum_{i=1}^{\ell_0} s_i \le \frac 1n\sum_{i=1}^n s_i.
$$
Therefore,
\begin{align}\label{sum2}
\sum_{i=1}^{\ell_0} s_i\le \frac {\ell_0}n\sum_{i=1}^n s_i.
\end{align}
As $\ell_0>n$, we have $s_0\le \frac {\ell_0}n s_0$.  
Combining this with (\ref{sum2}),
(\ref{sum}), and (\ref{l0}), we have
$$
\ell=\sum_{i=0}^{\ell_0} s_i\le \frac {\ell_0}n\sum_{i=0}^n s_i\le  \frac
{\ell_0}n\dim X\le \frac rn\dim X.\qedhere
$$
\end{proof}

We conclude with the analog of the Noguchi/Winkelmann theorem
for non-Archimedean analytic maps from $\aonecross.$  We remind the 
reader that in the complex case, the exponential map 
provides a non-constant complex analytic 
map from $\aone(\mathbf{C})$ to $\aonecross(\mathbf{C})$
and of course $\aonecross(\mathbf{C})\subset\aone(\mathbf{C}),$
so that in complex analytic geometry there is no difference in algebraic
degeneracy theorems from maps from $\aone(\mathbf{C})$ or
$\aonecross(\mathbf{C}).$  However, in the non-Archimedean case
the only analytic maps from $\aone$ to $\aonecross$ are the constants,
and thus it is more difficult to have an algebraically non-degenerate
analytic map from $\aone$ than it is to have an algebraically
non-degenerate analytic map from $\aonecross.$

\begin{theorem}\label{aonecrossthm}
Let $Y$ be a possibly singular projective variety that admits
a desingularization \hbox{$\widetilde{Y}\rightarrow Y,$}
and let $\iota:Y\rightarrow X$ be a morphism to a smooth projective
variety $X.$
Let $\{D_i\}_{i=1}^\ell$ be $\ell$ irreducible effective divisors on $X$
such that $\{\iota^* D_i\}_{i=1}^\ell$ form $\ell$ distinct effective
Cartier divisors on $Y.$ Let $a$ denote the dimension of the
Albanese variety of $Y.$ Let $r$ be the rank of the subgroup
generated by the $c_1(D_i)$ in $\NS(X).$ If
\hbox{$\ell > r+\dim Y - a,$}
then any analytic map from 
$\aonecross$ to $Y$ is either algebraically degenerate or
intersects the support of at least one of the $\iota^*D_i.$
\end{theorem}

\begin{remark*} In characteristic zero by Hironaka's Theorem,
any $Y$ admits a desingularization $\widetilde{Y}.$
Because resolution of singularities is not yet known in positive
characteristic, we make the existence of a desingularization an
explicit hypothesis.  Unlike in Theorem~\ref{aonethm}, here we will need
a morphism rather than a rational map to the Albanese variety, so working
with a normalization is not sufficient.
\end{remark*}

\begin{proof}
Let $\tilde\iota$ be the natural map from $\widetilde{Y}$ to $X$
induced by $\iota.$
Let $Y'$ be the variety obtained by deleting the supports of
$\tilde\iota^*D_i$ from $\widetilde{Y}.$ 
By \cite{Serre58a},
there is a morphism $\alpha$ from $Y'$ to a semi-Abelian variety $S$
such that $S$ is generated by the differences of points in the
image of $Y'$ and such that $S$ is the extension of the Albanese
variety $A$ of $\widetilde{Y}$ by a multiplicative torus.
As in \cite{Serre58b}, let $I$ denote the
free Abelian group generated by the $\tilde\iota^* D_i$ and let $J$
be the kernel of the mapping from $I$ to $\NS(\widetilde{Y}).$
Then, it follows from 
the discussion in \cite{Serre58b} that
the dimension of $S$ is the dimension of $A$ plus the rank of $J.$

Let $K$ be the subgroup of the free Abelian group generated by the 
$D_i$ that maps to zero in $\NS(X).$ By hypothesis, $K$ has rank
$\ell-r.$  Consider the map 
\hbox{$\tilde\iota^*:K\rightarrow J.$} By our assumption that the
$\iota^*D_i$ are distinct and the fact that $\widetilde{Y}$ maps
onto $Y,$ we have that the $\tilde\iota^*D_i$
are distinct. Hence, $\tilde\iota^*$
injects $K$ into $J,$ and thus $J$ has rank at least $\ell-r.$
So, if 
\hbox{$\ell > r+\dim Y - a,$} then
\hbox{$\dim S \ge \ell-r+ a > \dim Y.$}
If $f$ is an analytic map to $Y$ not contained in the singular locus
and not intersecting the supports of $\iota^* D_i,$
then $f$ lifts to an analytic map $\tilde f$ to $Y'.$
By Theorem~\ref{savsubv}, $\alpha\circ\tilde f$ is contained in the
translate of a semi-Abelian subvariety of $S$ contained in 
the proper subvariety $\alpha(Y').$
Because differences of points in $\alpha(Y')$ 
generate $S,$ the variety $\alpha(Y')$  
cannot be a translate of a proper semi-Abelian subvariety of $S,$
and hence $\tilde f$ and $f$ are algebraically degenerate.
\end{proof}

Applying the argument on pages~606--607 of Noguchi/Winkelmann,
\textit{i.e.,} replacing $\ell_0\le r$ with
$\ell_0 \le r+\dim Y - a$ in equation~(\ref{l0}) in the proof
of our Corollary~\ref{degencor}, then yields the following corollaries.

\begin{corollary}
Let $Y$ be a closed positive dimensional 
subvariety of a non-singular projective
variety $X$ admitting a desingularization 
\hbox{$\widetilde{Y}\rightarrow Y.$}
Let $a$ be the dimension of the 
Albanese variety of $Y.$ Let $\{D_i\}_{i=1}^\ell$ be $\ell$ irreducible,
effective, ample divisors in general position on $X.$ 
Let $r$ be the rank of the subgroup of $\NS(X)$
generated by $\{c_1(D_i)\}_{i=1}^\ell.$
If there exists an algebraically non-degenerate analytic map
from $\aonecross$ to $Y$ omitting each of the $D_i$ that does not
contain all of $Y,$
then 
$$
	(\ell - \dim X)\cdot \dim Y \le \dim X \cdot \max \{0,r-a\}
$$
\end{corollary}
\begin{corollary} 
Let $X$ be a non-singular projective
variety in characteristic zero.  Let $\{D_i\}_{i=1}^\ell$ be $\ell>\dim X$
irreducible, effective, ample divisors in general position on $X,$ 
and let $r$ be the rank in $\NS(X)$ of the subgroup generated
by $\{c_1(D_i)\}_{i=1}^\ell.$
Let $f$ be an analytic map from $\aonecross$ to $X$ omitting each of the
$D_i.$   Then the image of $f$  is contained in an algebraic subvariety
$Y$ of $X$ such that
$$
\dim Y\le \frac{r\cdot\dim X}{\ell-\dim X}
$$
In particular, if
$$
	\ell\ge(r+1)\cdot\dim X+1
$$
then $f$ is constant.
\end{corollary}

\textbf{Acknowledgements.} The authors would like to thank 
Jean-Luis Colliot-Th\'el\`ene for some comments about Chern classes
in \'etale cohomology. This paper was completed while the second 
and third authors
were visiting the Institute of Mathematics at the Vietnamese Academy of
Science and Technology, to which they would like to express their thanks
for the warm hospitality they receive there.



\begin{thebibliography}{[MM-XX]}
\bibitem[An 07]{AnProc} T.~T.~H.~An, A defect relation for
non-Archimedean analytic curves in arbitrary projective varieties,
Proc.\ Amer.\ Math.\ Soc.\  \textbf{135}  (2007),  1255--1261.
\bibitem[AWW 07]{AnWangWong} T.~T.~H.~An, J.~T.-Y.~Wang, and
P.-M.~Wong, Non-Archimedean analytic curves in the complements of 
hypersurface divisors, preprint, 2007.
\bibitem[Bl 26]{Bloch} A.~Bloch, Sur les syst\`emes de fonctions uniformes
satisfaisant \`a l'\'equation d'une vari\'et\'e alg\'ebrique dont
l'irr\'egularit\'e d\'epasse la dimension, J.\ Math.\ Pures
Appl.\ \textbf{5} (1926), 9--66.
\bibitem[Ch~93]{Cherry93} W.~Cherry, \textit{Hyperbolic $p$-Adic Analytic
Spaces,} Ph.D.\ Thesis, Yale University, 1993.
\bibitem[Ch~94]{Cherry94} W.~Cherry, Non-Archimedean analytic curves
in abelian varieties,  Math.\ Ann.\  \textbf{300}  (1994),   393--404.
\bibitem[CR~04]{CherryRu} W.~Cherry and M.~Ru,
Rigid analytic Picard theorems, Amer.\ J.\ Math.\ \textbf{126} (2004),
873--889.
\bibitem[Du 44]{Dufresnoy} H.~Dufresnoy,
Th\'eorie nouvelle des familles complexes normales.  Applications
\`a l'\'etude des fonctions alg\'ebro\"ides.
Ann.\ Sci.\ Ecole Norm.\ Sup.\ \textbf{61} (1944), 1--44.
\bibitem[Ig 55]{Igusa} J.~Igusa, On Some Problems in Abstract
Algebraic Geometry, Proc.\ N.\ A.\ S.\ \textbf{41} (1955), 964--967.
\bibitem[Ko 98]{Kobayashi} S.~Kobayashi, \textit{Hyperbolic Complex
Spaces,} Springer-Verlag, 1998
\bibitem[La 59]{LangAV} S.~Lang, \textit{Abelian Varieties,}
Interscience Publishers, Inc., 1959 (reprinted by Springer-Verlag in 1983).
\bibitem[La 83]{LangFDG} S.~Lang, \textit{Fundamentals of Diophantine 
Geometry,} Graduate Texts in Mathematics \textbf{191}, Springer-Verlag, 1983.
\bibitem[La 87]{LangHyp} S.~Lang, \textit{Introduction to Complex
Hyperbolic Spaces,} Springer-Verlag, 1987.
\bibitem[Mi 80]{Milne} J.~Milne, \textit{Etale Cohomology,}
Princeton Mathematical Series \textbf{33}, Princeton University Press,
1980.
\bibitem[NW 02]{NoguchiWinkelmann} J.~Noguchi and J.~Winkelmann,
Holomorphic curves and integral points off divisors, Math.\ Z.\ 
\textbf{239} (2002), 593--610.
\bibitem[O 77]{Ochiai} T.~Ochiai, On holomorphic curves in algebraic
varieties with ample irregularity, Invent.\ Math.\ \textbf{43} (1977),
83--96.
\bibitem[Ru 01]{RuHypersurface} M.~Ru, A note on $p$-adic Nevanlinna
theory,  Proc.\ Amer.\ Math.\ Soc.\  \textbf{129}  (2001),  1263--1269.
\bibitem[Se 58a]{Serre58a} J.~P.~Serre, Morphismes universels
et vari\'et\'e d'Albanese, S\'eminaire Claude Chevalley
\textbf{4} (1958-59), Expos\'e 10.
\bibitem[Se 58b]{Serre58b} J.~P.~Serre, Morphismes universels
et diff\'erentielles de troisi\`eme esp\`ece, S\'eminaire Claude Chevalley
\textbf{4} (1958-59), Expos\'e 11.
\end{thebibliography}
\end{document}